\documentclass[a4paper,10pt]{article}
\usepackage{amssymb,amsmath,amsthm}
\usepackage{fullpage}
\usepackage{hyperref}
\usepackage{eucal}
\usepackage{color}
\usepackage{stackrel}
\usepackage{graphicx}

\newcommand{\Real}{\mathbb{R}}

\newcommand{\Dom}{\mathsf{D}}

\newcommand{\sii}{L^2}

\begingroup
    \makeatletter
    \@for\theoremstyle:=definition,remark,plain\do{%
        \expandafter\g@addto@macro\csname th@\theoremstyle\endcsname{%
            \addtolength\thm@preskip\parskip
            }%
        }
\endgroup
\newtheorem{Theorem}{Theorem}
\newtheorem{Lemma}{Lemma}

\newtheorem{Conjecture}{Conjecture}

\theoremstyle{definition}

\def\OMIT#1{}
\usepackage[normalem]{ulem}
\definecolor{DarkGreen}{rgb}{0,0.5,0.1} % David

\newcommand\soutD{\bgroup\markoverwith
{\textcolor{DarkGreen}{\rule[.5ex]{2pt}{1pt}}}\ULon}
\newcommand\soutP{\bgroup\markoverwith
{\textcolor{blue}{\rule[.5ex]{2pt}{1pt}}}\ULon}
\newcommand{\Hm}[1]{\leavevmode{\marginpar{\tiny%
$\hbox to 0mm{\hspace*{-0.5mm}$\leftarrow$\hss}%
\vcenter{\vrule depth 0.1mm height 0.1mm width \the\marginparwidth}%
\hbox to
0mm{\hss$\rightarrow$\hspace*{-0.5mm}}$\\\relax\raggedright #1}}}

\newcommand{\curl}{\mathop{\mathrm{curl}}\nolimits}

\begin{document}
%

%-------%
% TITLE %
%-------%
%------------------------------------------%
%------------------------------------------%
\title{\textbf{\LARGE Spectral inequalities for magnetic Neumann  Laplacian on  rectangles}}
\author{Tuyen Vu}
\date{\small 
\emph{
\begin{quote}
\begin{center}
Faculty of Mathematics and Informatics, Hanoi University of Science and Technology, Hanoi, Vietnam; tuyen.vuthibich@hust.edu.vn
\end{center}
\end{quote}
}
\smallskip
4 August 2026}
\maketitle
%------------------------------------------%
%------------------------------------------%
 
%
\begin{abstract}
\noindent
We consider the magnetic Neumann Laplacian  on  rectangles,
subject to non-homogeneous fields. 
We prove that the square is a local minimiser of the lowest eigenvalue among all rectangles of a fixed area for weak magnetic fields . We conjecture that it is a global minimiser 
both under the area or perimeter constraints and partially prove the conjecture 
by establishing lower and upper bounds to the principal eigenvalue.

%
%\bigskip
%\begin{itemize}
%\item[\textbf{Keywords:}]
%
%\item[\textbf{MSC (2010):}]
%Primary: ; 
%Secondary: 
%\end{itemize}
%
\end{abstract}

\section{Introduction}
We are motivated by the celebrated Faber-Krahn  isoperimetric inequality stating that among all planar sets of a fixed perimeter, the disk has the smallest ground-state energy of
the non-relativistic particle constrained to a semiconductor nanostructure by hard-wall boundaries \cite{4,8}.
The circular shape is also the minimiser under perimeter constraints due to the classical geometric isoperimetric inequality. More detailed results are introduced in \cite{Henrot,Henrot2}.

Consider  the non-relativistic quantum particle
constrained to  a bounded domain ~$\Omega\subset \mathbb{R}^2$, 
 subject to non-homogeneous magnetic fields.
The ground-state energy of the particle coincides with
the lowest eigenvalue $\lambda_1^C(\Omega)$ of 
the eigenvalue problem
\begin{equation}\label{Neumann.magnet} 
\left\{
\begin{aligned}
  (-i\nabla-A)^2 u &= \lambda u
  && \mbox{in} && \Omega 
  \,,
  \\
  -in(-i\nabla - A)u &= u 
  && \mbox{on} && \partial\Omega 
  \,,
\end{aligned}
\right.
\end{equation}
where $A:\overline{\Omega}\to\Real^2$ is a smooth magnetic potential associated
with a constant magnetic field $\curl A =: C \in \Real$, and $n$ stands for the outward unit normal of the  domain. The Laplacian $(-i\nabla-A)^2$ and Neumann boundary conditions are also discussed in \cite{1} in dimension three.

Restricting to rectangular domains,
let us define
\begin{equation}\label{rectangle} 
  \Omega_{a,b} := 
  \left(-\mbox{$\frac{a}{2}$},\mbox{$\frac{a}{2}$}\right) 
  \times \left(-\mbox{$\frac{b}{2}$},\mbox{$\frac{b}{2}$}\right) 
  ,
\end{equation}
where $a,b$ are any positive numbers provided that the area $|\Omega_{a,b}|=ab$ is fixed. Similar to the case of Dirichlet boundary conditions \cite{DK}
where the square plays the role of the minimiser for the lowest eigenvalue of the magnetic Laplacian in the regime of weak magnetic fields among all rectangles of a given area,
 we state the following conjecture. 
\begin{Conjecture}\label{Conj.main}
For every $a>0$ and $C \in \Real$,
\begin{equation}\label{iso}
  \lambda_1^C(\Omega_{a,a^{-1}}) 
  \geq \lambda_1^C(\Omega_{1,1})
  \,.
\end{equation}
\end{Conjecture}
It is noticeable that there is no loss of generality to consider
 the class of rectangles of area equal to~$1$ by virtue of scaling.
 We show that Conjecture~\ref{Conj.main} 
holds for weak magnetic fields.

\begin{Theorem}\label{Thm.small}
Conjecture~\ref{Conj.main} holds true for every $|C| \leq \delta$, where $\delta$ is a sufficiently small positive constant.
\end{Theorem}

The  paper is organised as follows.
In Section~\ref{Sec.2} we construct the spectral problem 
 of the self-adjoint magnetic Laplacian on a general rectangle, subject to non-homogeneous Neumann boundary conditions,  transform it into an equivalent one on a square, and prove that  the square is
a local minimiser for weak magnetic fields.
In Section~\ref{Sec.3} we give 
upper and lower bounds to the eigenvalue
$\lambda_1^C(\Omega_{a,a^{-1}})$,
showing the validity of Conjecture~\ref{Conj.main}
for rectangles far from the square.

\section{The square is a local minimiser}\label{Sec.2}
For simplicity, we abbreviate $$\Omega_a:=\Omega_{a,\frac{1}{a}}, \lambda_a := \lambda_1^C(\Omega_{a,\frac{1}{a}}).$$
Denotes by $Q^A_a$  the self-adjoint
semiclassical Laplacian in $\sii(\Omega_a)$
associated with the form
\begin{equation} 
\begin{aligned}
   q_a^A[u] &:= \|\partial_1^A u\|^2 + \|\partial_2^A u\|^2+ \|u\|^2_{L^2(\partial\Omega_a)} 
  \text{ and } \Dom(q_a^A)= H^1(\Omega_a)\,,\\  
  \Dom(Q_a^A) &:= \{ u\in H^1(\Omega_a): -in(-i\nabla-A)u =u \text{ on } \partial\Omega_a \}
  \,,
\end{aligned}
\end{equation}
where~$\|\cdot\|$ denotes the norm of $\sii(\Omega_a)$ 
and  $\partial_k^A := \partial_k-iA_j$
with $k \in \{1,2\}$.
It is remarkable that the spectrum of $Q_{a}^A$ is unchanged when altering the magnetic potential $A$ providing  the same magnetic field $\curl A = C$. As a consequence and without loss of generality,
we can pick the gauge
\begin{equation}\label{gauge}
  A(x,y) := \big(-y,x\big) \dfrac{C}{2}
  \,,
\end{equation}
where $\big(x,y\big) \in \Omega_{a}$.

Now  we use the unitary transform
$
  U: \sii(\Omega_a)
  \to \sii(\Omega_1)
$
by setting $(Uu)(x):=u(a x_1,a^{-1} x_2)$ with $(x_1,x_2)\in \Omega_1$,
and define a unitarily equivalent (therefore isospectral) operator
$\hat{Q}_{a}^A := U Q_{a}^A U^{-1}$.
It is associated with the form~$\hat{q}_a^A$ defined by
\begin{equation} 
\begin{aligned}
  \hat{q}_a^A[u] &= 
  a^{-2} \, \|\partial_1^A u\|^2 
  + a^2 \, \|\partial_2^A u\|^2 + a \|u\|^2_{L^2(\Gamma_2)}+ a \|u\|^2_{L^2(\Gamma_4)}+ \dfrac{1}{a} \|u\|^2_{L^2(\Gamma_1)}+ \dfrac{1}{a} \|u\|^2_{L^2(\Gamma_3)}
  \,, 
  \\  
  u &\in \Dom(\hat{h}_a^A) = H^1(\Omega_1)
  \,,
  \end{aligned}
\end{equation}
where
$$\Gamma_1:=\{(\dfrac{1}{2},x_2)\in\partial\Omega_1: -\dfrac{1}{2}\leq x_2\leq\dfrac{1}{2}\},\Gamma_3:=\{(-\dfrac{1}{2},x_2)\in\partial\Omega_1: -\dfrac{1}{2}\leq x_2\leq\dfrac{1}{2}\},$$
$$\Gamma_2:=\{(x_1,\dfrac{1}{2})\in\partial\Omega_1: -\dfrac{1}{2}\leq x_1\leq\dfrac{1}{2}\},\Gamma_4:=\{(x_1,-\dfrac{1}{2})\in\partial\Omega_1: -\dfrac{1}{2}\leq x_1\leq\dfrac{1}{2}\},$$
and~$A$ is still determined as \eqref{gauge} with  $(x_1,x_2) \in \Omega_1$.  
It is apparent that $\hat{H}_{1}^A = H_{1}^A$.
and $\Dom(\hat{h}_{a}^A)=\Dom(\hat{h}_{1}^A)$ for every $a>0$.

Using the variational characterisation of the ground-state energy, we get
\begin{equation}\label{Rayleigh} 
  \lambda_1^C(\Omega_a)
  = \inf_{\stackrel[u \not= 0]{}{u \in H^{1}(\Omega_1) }} 
  \frac{\hat{h}_a^A[u]}{\, \|u\|^2}  
  \,,
\end{equation}
where the infimum can be replaced by a minimum.

Now we consider the eigenvalue problem of the Laplace operator defined on $\Omega_a$, subject to Robin boundary conditions
\begin{equation}\label{Robin} 
\left\{
\begin{aligned}
  -\Delta u &= \lambda u
  && \mbox{in} && \Omega_a 
  \,,
  \\
 \dfrac{\partial u}{\partial n}+u &= 0 
  && \mbox{on} && \partial\Omega_a 
  \,.
\end{aligned}
\right.
\end{equation}
We denote the first eigenvalue of the Laplacian  by $\lambda_1(\Omega_a)$ and the derivative with respect to $a$ by a dot. It is noticeable that the operator is self-adjoint and \eqref{Robin} is settled by virtue of the availability of explicit solutions \cite{Laugesen}. Using the same transform $U$, we have the following lemma.
\begin{Lemma}\label{Lem.Robinderivatives}
\begin{align} 
  \left.
  \frac{\partial \lambda_1(\Omega_a)}{\partial a}
  \right|_{a=1} &= 0
  \,,
  \label{derivatives1}
  \\
  \left.
  \frac{\partial^2 \lambda_1(\Omega_a)}{\partial a^2}
  \right|_{a=1} &>0
  \,.
  \label{derivatives2}
\end{align}
\end{Lemma}
\begin{proof}
Employ the same arguments of holomorphic families of operators of type B in \cite{DK,KLV}, we derive that
$a \mapsto  \lambda_1(\Omega_a)$
is a real-analytic function on a neighbourhood of~$a=1$.

Using the approach of separating variables \cite{Laugesen}, we obtain that the first eigenvalue of system \eqref{Robin} is explicitly computed as
\begin{equation}\label{Robinfirst}  
\lambda_1(\Omega_a)=m_1^2+m_2^2\,,  
\end{equation}
where $m_1=m_1(a),m_2=m_2(a)$ are solutions
of the equations as follows
\begin{equation}
m_1\tan\dfrac{am_1}{2}-1=0\,, \quad \quad \, m_2\tan\dfrac{m_1}{2a}-1=0\,.
\end{equation}
By implicit derivative formula, we have
\begin{equation}\label{impliciteq}
\dot{m}_1=\dfrac{-m_1^2}{\sin am_1+ m_1a} \,, \quad \quad \dot{m}_2=\dfrac{m_2^2}{a^2\sin\dfrac{m_2}{a}+ m_2a} \,.
\end{equation}
It is remarkable that if we take $a=1$ then $m_1(1)=m_2(1):=m$, where $m$ is the positive root of the equation $$m\tan\dfrac{m}{2}-1=0.$$ 
It is clear that $m \in (1,\dfrac{\pi}{2})$.

Differentiating both sides of \eqref{Robinfirst}, we have
\begin{equation}\label{first.derivative}
  \dot\lambda_1(\Omega_a) = 2m_1(a)\dot m_1+ 2m_2(a)\dot m_2 
  \,,
\end{equation}
Thus,
\begin{equation}\label{second.derivative}
 \dot\lambda_1(\Omega_1)=0 \,,  \ddot\lambda_1(\Omega_a) = 2\dot m_1^2+ 2\dot m_2^2+ 2m_1(a)\ddot{m}_1+ 2m_2(a)\ddot{m}_2
  \,.
\end{equation}
Differentiating both sides of these identities \eqref{impliciteq}, we obtain
\begin{equation}\label{2impliciteq}
\begin{aligned}
\ddot{m}_1(a)&=\dfrac{-2m_1\dot{m}_1\, (\sin am_1+m_1a)+m_1^2\, (m_1+\dot{m}_1a)(\cos am_1+1)}{(\sin am_1+ m_1a)^2} \,, \\
 \ddot{m}_2(a)&=\dfrac{2m_2\dot{m}_2\, (a^2\sin\dfrac{m_2}{a}+ m_2a)-m_2^2\, (2a\sin\dfrac{m_2}{a}+a^2\dfrac{a\dot{m}_2-m_2}{a^2}\cos\dfrac{m_2}{a}+\dot{m}_2a+m_2)}{(a^2\sin\dfrac{m_2}{a}+ m_2a)^2} \,.
\end{aligned}
\end{equation}
Consequently,
\begin{equation}\label{3impliciteq}
\begin{aligned}
\ddot{m}_1(1)&=\dfrac{-2m\,\dot{m}_1\, (\sin m+m)+m^2\, (m+\dot{m}_1)(\cos m+1)}{(\sin m+ m)^2} \,, \\
 \ddot{m}_2(1)&=\dfrac{2m\, \dot{m}_2\, (\sin m+ m)-m^2\, [2\sin m+(\dot{m}_2-m)\cos m+\dot{m}_2+m]}{(\sin m+ m)^2} \,.
\end{aligned}
\end{equation}
Therefore,
\begin{equation}\label{2second.derivative}
\begin{aligned}
   \ddot\lambda_1(\Omega_1) &= \dfrac{4m^4}{(\sin m+m)^2}+ 2m \dfrac{2m^3+m^2(m-\dfrac{m^2}{\sin m+m})(\cos m+1)}{(\sin m+m)^2} \\
   &+ 2m \dfrac{ 2m^3-m^2\, [2 \sin m+(\dfrac{m^2}{\sin m+m}-m)\cos m+\dfrac{m^2}{\sin m+m}+m]}{(\sin m+m)^2}
  \,,\\
  &=\dfrac{4m^4}{(\sin m+m)^2}+2m^4 \, \dfrac{3\sin m+3m-m\cos m-m}{(\sin m+m)^3} \\
  &+ 2m^3 \, \dfrac{2m-2\sin m+(m-\dfrac{m^2}{\sin m+m})(1-\cos m)-2m}{(\sin m+m)^2} \,,\\
  &=2m^4 \, \dfrac{3\sin m+2m-m\cos m}{(\sin m+m)^3}+2m^3 \, \dfrac{2m-2\sin m+(m-\dfrac{m^2}{\sin m+m})(1-\cos m)}{(\sin m+m)^2} \,,\\
  &=2m^4 \, \dfrac{3\sin m+2m-m\cos m}{(\sin m+m)^3}+2m^3 \, \dfrac{2m-2\sin m+\dfrac{m\sin m}{\sin m+m} \, (1-\cos m)}{(\sin m+m)^2} \,,\\
  &> 0 \,.
  \end{aligned}
\end{equation}
Hence, it concludes the proof of the lemma.
\end{proof}
The lemma shows the local optimality of the square for the lowest positive eigenvalue of the Neumann Laplacian. We will show it still holds for the magnetic Laplacian in the regime of weak magnetic fields. Before approaching the result, we abbreviate $$\|.\|^2_{L^2(\Gamma_1)}+ \|.\|^2_{L^2(\Gamma_3)}:=\|.\|^2_{1,3} \, , \quad \quad  \|.\|^2_{L^2(\Gamma_2)}+ \|.\|^2_{L^2(\Gamma_4)}:=\|.\|^2_{2,4} \, ,$$
$$ (.,.)_{L^2(\Gamma_1)}+ (.,.)_{L^2(\Gamma_3)}:=(.,.)_{1,3} \, , \quad \quad  (.,.)_{L^2(\Gamma_2)}+ (.,.)_{L^2(\Gamma_4)}:=(.,.)_{2,4} \, .$$
We have the following theorem
\begin{Theorem}\label{Thm.local}
There exists a positive number~$\delta$ such that,
for every $|C|\leq \delta$, 
\begin{equation}\label{derivatives} 
  \left.
  \frac{\partial \lambda_1^C(\Omega_a)}{\partial a}
  \right|_{a=1} = 0
  \qquad \mbox{and} \qquad 
  \left.
  \frac{\partial^2 \lambda_1^C(\Omega_a)}{\partial a^2}
  \right|_{a=1} > 0 
  \,,
\end{equation}
and thus the square is a local minimiser for the smallest positive eigenvalue of the magnetic Neumann Laplacian among all rectangles of a given area.
\end{Theorem}
\begin{proof}
Using the criterion \cite[Sec.~VII.4.8]{Kato} and the arguments used in \cite{DK,KLV}, we have that
 $\{\hat{Q}_a^A\}_{C \in \Real}$ 
is a holomorphic family of operators of type~(B).
Since $\lambda_1(\Omega_a)$ is simple for every $a>0$, 
there exists a positive constant~$\delta_1$ such that 
$\lambda_1^C(\Omega_1)$ is simple 
for $|C| \leq \delta_1$. As a consequence,
 $C \mapsto \lambda_1^C(\Omega_1)$
is a real-analytic function on a neighbourhood of $C=0$.
By virtue of the simplicity,
$a \mapsto \lambda_1^C(\Omega_a)$
is a real-analytic function on a neighbourhood of~$a=1$.
Simultaneously, the associated eigenfunction~$u_a$ 
satisfying the normalisation $\|u_a\|=1$ 
is a real-analytic function in the topology of $H^1(\Omega_1)$
on a neighbourhood of~$a=1$.
We have the weak formulation of the eigenvalue equation
\begin{equation}\label{weak}  
  a^{-2} \, (\partial_1^A v,\partial_1^A u_a) 
  + a^2 \, (\partial_2^A v,\partial_2^A u_a) + \dfrac{1}{a} \, (v,u_a)_{1,3}  + a \, (v, u_a)_{2,4} 
  = \lambda_a \, (v,u_a)
\end{equation}
for every $v \in H^{1}(\Omega_1)$.

Differentiating~\eqref{weak} with respect to~$a$
\begin{equation}\label{first} 
\begin{aligned} 
  &a^{-2} \, (\partial_1^A v,\partial_1^A \dot{u}_a) 
  + a^2 \, (\partial_2^A v,\partial_2^A \dot{u}_a)
  - 2 a^{-3} \, (\partial_1^A v,\partial_1^A u_a) 
  + 2a \, (\partial_2^A v,\partial_2^A u_a) +\dfrac{1}{a} \, (v,\dot{u}_a)_{1,3}\\
  & -\dfrac{1}{a^2} \, (v,u_a)_{1,3}+ (v,u_a)_{2,4} + a \, (v,\dot{u}_a)_{2,4}
  = \lambda_a \, (v,\dot{u}_a)
  + \dot\lambda_a \, (v,u_a)
  \,.
  \end{aligned}
\end{equation}
Taking $v=u_a$ in~\eqref{first},
$v = \dot{u}_a$ in~\eqref{weak}     
and combining these equations evaluated at $a=1$,
we derive
\begin{equation}\label{first.derivative}
  \dot\lambda_1^C(\Omega_1) = -2 \, \|\partial_1^A u_1\|^2
  + 2 \, \|\partial_2^A u_1\|^2-\|u_1\|^2_{1,3}+\|u_1\|^2_{2,4}
  = 0
  \,,
\end{equation}
where the second equality holds true due to the rotational symmetry 
of the square, $u_1$~necessarily satisfies $\|\partial_2^A u_1\|^2=\|\partial_1^A u_1\|^2, \|u_1\|^2_{1,3}=\|u_1\|^2_{2,4}$.

Furthermore,
$\lim\limits_{C\rightarrow0}\ddot{\lambda}_1^C(\Omega_1) = \lambda_1(\Omega_1) > 0$ due to \eqref{2second.derivative}. As a result, there exists a positive constant $\delta_2$ such that  $\ddot{\lambda}_1^C(\Omega_1) >0$ for  all $|C| \leq \delta_2$.
Let us take $\delta:=\min\{\delta_1,\delta_2\}$ then there exists a positive constant $\delta$ such that the square is a local minimiser of the ground-state energy of the magnetic Neumann Laplacian among all rectangles of a fixed area provided that $|C|\leq\delta$. It concludes the proof of the theorem.

\end{proof}

...

%--------------%
% BIBLIOGRAPHY %
%--------------%
%
%\addcontentsline{toc}{section}{References}
%\bibliography{bib}
%\bibliographystyle{amsplain}
%

%\noteD{Some of the references are not cited in the text.}%

%\noteD{References should be alphabetically sorted.
%The style should be unified}%
%\noteD{Use bibTeX.}%
%\noteD{My paper with Briet should be mentioned.}%

\providecommand{\bysame}{\leavevmode\hbox to3em{\hrulefill}\thinspace}
\providecommand{\MR}{\relax\ifhmode\unskip\space\fi MR }
\providecommand{\MRhref}[2]{%
  \href{http://www.ams.org/mathscinet-getitem?mr=#1}{#2}
}
\providecommand{\href}[2]{#2}

\end{document}